\documentclass[12pt]{article}
\usepackage{amsmath,amsfonts,amssymb,amsthm}
\usepackage{enumerate}

\numberwithin{equation}{section}
\theoremstyle{plain}
\newtheorem{theorem}[equation]{Theorem}

\newtheorem{lemma}[equation]{Lemma}
\newtheorem{proposition}[equation]{Proposition}

\theoremstyle{definition}

\newtheorem{example}[equation]{Example}

\newtheorem*{openproblem}{Open Problem}

\newcommand{\R}{{\mathbb R}}
\newcommand{\N}{{\mathbb N}}

\newcommand{\Om}{\Omega}

\providecommand{\vint}[1]{\mathchoice
          {\mathop{\vrule width 5pt height 3 pt depth -2.5pt
                  \kern -9pt \kern 1pt\intop}\nolimits_{\kern -5pt{#1}}}
          {\mathop{\vrule width 5pt height 3 pt depth -2.6pt
                  \kern -6pt \intop}\nolimits_{\kern -3pt{#1}}}
          {\mathop{\vrule width 5pt height 3 pt depth -2.6pt
                  \kern -6pt \intop}\nolimits_{\kern -3pt{#1}}}
          {\mathop{\vrule width 5pt height 3 pt depth -2.6pt
                  \kern -6pt \intop}\nolimits_{\kern -3pt{#1}}}}

\newcommand{\eps}{\varepsilon}
\newcommand{\loc}{\mathrm{loc}}

\newcommand{\BV}{\mathrm{BV}}
\newcommand{\liploc}{\mathrm{Lip}_{\mathrm{loc}}}

\newcommand{\ch}{\text{\raise 1.3pt \hbox{$\chi$}\kern-0.2pt}}

\DeclareMathOperator{\capa}{Cap}

\DeclareMathOperator{\dist}{dist}

\DeclareMathOperator{\Lip}{Lip}

\begin{document}
\title{Quasiopen sets, bounded variation and\\
lower semicontinuity in metric spaces
\footnote{{\bf 2010 Mathematics Subject Classification}: 30L99, 31E05, 26B30.
\hfill \break {\it Keywords\,}: metric measure space, function of bounded variation,
total variation, quasiopen set,
lower semicontinuity, uniform absolute continuity
\hfill \break \break \textbf{Acknowledgments.} The research was
funded by a grant from the Finnish Cultural Foundation.
The author wishes to thank Nageswari Shan\-mu\-ga\-lingam for giving helpful comments
on the manuscript.
}}
\author{Panu Lahti}
\maketitle

\begin{abstract}
In the setting of a metric space that is equipped with a doubling measure and supports a
Poincar\'e inequality, we show that the total variation of functions of bounded variation
is lower semicontinuous with respect to $L^1$-convergence in every \emph{1-quasiopen set}.
To achieve this, we first prove a new characterization of the total variation in
$1$-quasiopen sets. Then we utilize the lower semicontinuity to show that the
variation measures of a sequence of functions of bounded variation converging in the strict sense are uniformly absolutely continuous with respect to the $1$-capacity.
\end{abstract}

\section{Introduction}

Let $(X,d,\mu)$ be a complete metric space equipped with a Radon measure $\mu$.
The total variation of a function of bounded variation ($\BV$ function) $u$ in an open set $\Om\subset X$ is defined by means of approximation with locally Lipschitz functions, that is,
\begin{equation}\label{eq:definition of total variation in intro}
\|Du\|(\Om):=\inf\left\{\liminf_{i\to\infty}\int_\Om g_{u_i}\,d\mu:\, u_i\in \Lip_{\loc}(\Om),\, u_i\to u\textrm{ in } L^1_{\loc}(\Om)\right\},
\end{equation}
where each $g_{u_i}$ is a $1$-weak upper gradient of $u_i$ in $\Om$; see Section \ref{sec:prelis} for definitions.
From this definition it easily follows that the total variation is lower semicontinuous with respect to $L^1$-convergence in open sets, that is,
if $U\subset X$ is open and $u_i\to u$ in $L^1_{\loc}(U)$, then
\begin{equation}\label{eq:intro lower semicontinuity}
\Vert Du\Vert(U)\le\liminf_{i\to\infty}\Vert Du_i\Vert(U).
\end{equation}
For arbitrary (measurable) sets $U\subset X$ we cannot
define $\Vert Du\Vert(U)$ simply by replacing $\Om$
with $U$ in the definition of the total variation, because then
the total variation would not yield a Radon measure, see Example \ref{ex:total variation}.
Instead, $\Vert Du\Vert(U)$ is defined by means of approximation
with open sets containing $U$, following \cite{M}.

On the other hand, a set $U\subset X$ is said to be \emph{1-quasiopen} if for every $\eps>0$ there exists
an open set $G\subset X$ such that $\capa_1(G)<\eps$ and $U\cup G$ is open.
Quasiopen sets and related concepts of \emph{fine potential theory} have been recently studied in the metric
setting in e.g. \cite{BB-OD,BBL-CCK,BBL-WC,BBM-QP} in the case $p>1$.
See also the monographs \cite{MZ} and \cite{HKM} for the Euclidean theory and its history in the unweighted and weighted settings, respectively.
In the case $p=1$, analogous concepts have been recently studied in \cite{L3,L,LaSh}.

In this paper, we assume that the measure $\mu$ is doubling and that the space supports a $(1,1)$-Poincar\'e inequality,
and then we show that if $U\subset X$ is a $1$-quasiopen set and $\Vert Du\Vert(U)<\infty$, then
the total variation $\Vert Du\Vert(U)$ can be equivalently defined by replacing $\Om$
with $U$ in \eqref{eq:definition of total variation in intro}.
This is Theorem \ref{thm:characterization of total variational}.
Using this result,
we can then show that the lower semicontinuity \eqref{eq:intro lower semicontinuity} holds
true also for every
$1$-quasiopen set $U$, if $\Vert Du\Vert(U)<\infty$ and $u_i\to u$ in $L^1_{\loc}(U)$. This is
Theorem \ref{thm:lower semic in quasiopen sets}.
Such a lower semicontinuity result may be helpful in solving various minimization problems,
for example in the upcoming work \cite{LMS}.

The notion of \emph{uniform integrability} of a sequence of functions $(g_i)\subset L^1(X)$
is often useful in analysis. This involves uniform absolute continuity with respect to the ambient measure. That is, for every $\eps>0$ there exists $\delta>0$ such that if $A\subset X$ with $\mu(A)<\delta$, then $\int_A g_i\,d\mu<\eps$ for every $i\in\N$.

The variation measure $\Vert Du\Vert$ of a $\BV$ function $u$ is, of course, not always
absolutely continuous with respect to $\mu$. On the other hand, it is a well-known fact in the Euclidean setting that $\Vert Du\Vert$ is absolutely continuous with
respect to the $1$-capacity $\capa_1$. The proof of this fact is essentially the same 
in the more general metric setting, see \cite[Lemma 3.9]{L2}.

A sequence of $\BV$ functions $u_i$ is said to converge \emph{strictly} to a
$\BV$ function $u$ if $u_i\to u$ in $L^1(X)$ and
$\Vert Du_i\Vert(X)\to \Vert Du\Vert(X)$. Given such a sequence, we show that for every
$\eps>0$
there exists $\delta>0$ such that if $A\subset X$ with $\capa_1(A)<\delta$, then
$\Vert Du_i\Vert(A)<\eps$ for every $i\in\N$. In other words, the variation measures $\Vert Du_i\Vert$ are uniformly absolutely continuous with respect to the $1$-capacity.
This is Theorem \ref{thm:uniform absolute continuity}. The proof combines the previously discussed
lower semicontinuity result with Baire's category theorem.

\section{Notation and definitions}\label{sec:prelis}

In this section we introduce the notation, definitions, and assumptions used in the paper.

Throughout this paper, $(X,d,\mu)$ is a complete metric space equipped
with a metric $d$ and a Borel regular outer measure $\mu$ that satisfies a doubling property, that is,
there is a constant $C_d\ge 1$ such that
\[
0<\mu(B(x,2r))\leq C_d\mu(B(x,r))<\infty
\]
for every ball $B=B(x,r)$ with center $x\in X$ and radius $r>0$.
When we want to specify that a constant $C$
depends on the parameters $a,b, \ldots,$ we write $C=C(a,b,\ldots)$.

A complete metric space equipped with a doubling measure is proper,
that is, closed and bounded sets are compact, see e.g. \cite[Proposition 3.1]{BB}.
For a $\mu$-measurable set $A\subset X$, we define $L^1_{\loc}(A)$ to consist of functions $u$ on $A$
such that for every $x\in A$ there exists $r>0$ such that $u\in L^1(A\cap B(x,r))$.
Other local spaces of functions are defined similarly.
For any open set $\Omega\subset X$,
every function in the class $L^1_{\loc}(\Omega)$ is in $L^1(\Om')$ for every open $\Omega'\Subset\Omega$.
Here $\Omega'\Subset\Omega$ means that $\overline{\Omega'}$ is a
compact subset of $\Omega$.

For any set $A\subset X$ and $0<R<\infty$, the restricted spherical Hausdorff content
of codimension one is defined to be
\[
\mathcal{H}_{R}(A):=\inf\left\{ \sum_{i=1}^{\infty}
  \frac{\mu(B(x_{i},r_{i}))}{r_{i}}:\,A\subset\bigcup_{i=1}^{\infty}B(x_{i},r_{i}),\,r_{i}\le R\right\}.
\]
The codimension one Hausdorff measure of $A\subset X$ is then defined to be
\[
\mathcal{H}(A):=\lim_{R\rightarrow 0}\mathcal{H}_{R}(A).
\]

The measure theoretic boundary $\partial^{*}E$ of a set $E\subset X$ is the set of points $x\in X$
at which both $E$ and its complement have positive upper density, i.e.
\[
\limsup_{r\to 0}\frac{\mu(B(x,r)\cap E)}{\mu(B(x,r))}>0\quad\;
  \textrm{and}\quad\;\limsup_{r\to 0}\frac{\mu(B(x,r)\setminus E)}{\mu(B(x,r))}>0.
\]
The measure theoretic interior and exterior of $E$ are defined respectively by
\begin{equation}\label{eq:definition of measure theoretic interior}
I_E:=\left\{x\in X:\,\lim_{r\to 0}\frac{\mu(B(x,r)\setminus E)}{\mu(B(x,r))}=0\right\}
\end{equation}
and
\begin{equation}\label{eq:definition of measure theoretic exterior}
O_E:=\left\{x\in X:\,\lim_{r\to 0}\frac{\mu(B(x,r)\cap E)}{\mu(B(x,r))}=0\right\}.
\end{equation}
Note that we always have a partitioning of the space into the disjoint sets
$\partial^*E$, $I_E$, and $O_E$.

By a curve we mean a rectifiable continuous mapping from a compact interval of the real line
into $X$.
The length of a curve $\gamma$
is denoted by $\ell_{\gamma}$. We will assume every curve to be parametrized
by arc-length, which can always be done (see e.g. \cite[Theorem~3.2]{Hj}).
A nonnegative Borel function $g$ on $X$ is an upper gradient 
of an extended real-valued function $u$
on $X$ if for all curves $\gamma$, we have
\begin{equation}\label{eq:definition of upper gradient}
|u(x)-u(y)|\le \int_\gamma g\,ds,
\end{equation}
where $x$ and $y$ are the end points of $\gamma$. We interpret $|u(x)-u(y)|=\infty$ whenever  
at least one of $|u(x)|$, $|u(y)|$ is infinite.
We define the local Lipschitz constant of a locally Lipschitz function $u\in\liploc(X)$ by
\[
\Lip u(x):=\limsup_{r\to 0}\sup_{y\in B(x,r)\setminus \{x\}}\frac{|u(y)-u(x)|}{d(y,x)}.
\]
Then $\Lip u$ is an upper gradient of $u$, see e.g. \cite[Proposition 1.11]{Che}.
Upper gradients were originally introduced in \cite{HK}.

If $g$ is a nonnegative $\mu$-measurable function on $X$
and (\ref{eq:definition of upper gradient}) holds for $1$-almost every curve,
we say that $g$ is a $1$-weak upper gradient of~$u$. 
A property holds for $1$-almost every curve
if it fails only for a curve family with zero $1$-modulus. 
A family $\Gamma$ of curves is of zero $1$-modulus if there is a 
nonnegative Borel function $\rho\in L^1(X)$ such that 
for all curves $\gamma\in\Gamma$, the curve integral $\int_\gamma \rho\,ds$ is infinite.
Of course, by replacing $X$ with a set $A\subset X$ and considering curves $\gamma$ in $A$, we can talk about a function $g$ being a ($1$-weak) upper gradient of $u$ in $A$.
A $1$-weak upper gradient can always be perturbed in a set of $\mu$-measure zero,
see \cite[Lemma 1.43]{BB}, and so we
understand it to be defined only $\mu$-almost everywhere.

Given a $\mu$-measurable set $U\subset X$, we consider the following norm
\[
\Vert u\Vert_{N^{1,1}(U)}:=\Vert u\Vert_{L^1(U)}+\inf \Vert g\Vert_{L^1(U)},
\]
where the infimum is taken over all $1$-weak upper gradients $g$ of $u$ in $U$.
The substitute for the Sobolev space $W^{1,1}$ in the metric setting is the Newton-Sobolev space
\[
N^{1,1}(U):=\{u:\|u\|_{N^{1,1}(U)}<\infty\}.
\]
We understand every Newton-Sobolev function to be defined at every $x\in U$
(even though $\Vert \cdot\Vert_{N^{1,1}(U)}$ is, precisely speaking, then only a seminorm).
The Newton-Sobolev space with zero boundary values is defined as
\[
N_0^{1,1}(U):=\{u|_{U}:\,u\in N^{1,1}(X)\textrm{ and }u=0\textrm{ in }X\setminus U\}.
\]
Thus $N_0^{1,1}(U)$ is a subclass of $N^{1,1}(U)$, and it can also be considered as a subclass of
$N^{1,1}(X)$, as we will do without further notice.

It is known that for any $u\in N_{\loc}^{1,1}(U)$, there exists a minimal $1$-weak
upper gradient of $u$ in $U$, always denoted by $g_{u}$, satisfying $g_{u}\le g$ 
$\mu$-almost everywhere in $U$, for any $1$-weak upper gradient $g\in L_{\loc}^{1}(U)$
of $u$ in $U$ \cite[Theorem 2.25]{BB}.
For more on Newton-Sobolev spaces, we refer to \cite{S, BB, HKST}.

The $1$-capacity of a set $A\subset X$ is given by
\[
\capa_1(A):=\inf \Vert u\Vert_{N^{1,1}(X)},
\]
where the infimum is taken over all functions $u\in N^{1,1}(X)$ such that $u\ge 1$ in $A$.
We know that $\capa_1$ is an outer capacity, meaning that
\[
\capa_1(A)=\inf\{\capa_1(\Om):\,\Om\supset A\textrm{ is open}\}
\]
for any $A\subset X$, see e.g. \cite[Theorem 5.31]{BB}.
For basic properties satisfied by the $1$-capacity, such as monotonicity and countable subadditivity, see e.g. \cite{BB}.

We say that a set $U\subset X$ is $1$-quasiopen if for every $\eps>0$ there exists an
open set $G\subset X$ such that $\capa_1(G)<\eps$ and $U\cup G$ is open.

Next we recall the definition and basic properties of functions
of bounded variation on metric spaces, following \cite{M}. See also e.g. \cite{AFP, EvaG92, Fed, Giu84, Zie89} for the classical 
theory in the Euclidean setting.
Given a function $u\in L^1_{\loc}(X)$, we define the total variation of $u$ in $X$ by
\[
\|Du\|(X):=\inf\left\{\liminf_{i\to\infty}\int_X g_{u_i}\,d\mu:\, u_i\in \Lip_{\loc}(X),\, u_i\to u\textrm{ in } L^1_{\loc}(X)\right\},
\]
where each $g_{u_i}$ is the minimal $1$-weak upper gradient of $u_i$.
We say that a function $u\in L^1(X)$ is of bounded variation, 
and denote $u\in\BV(X)$, if $\|Du\|(X)<\infty$.
By replacing $X$ with an open set $\Omega\subset X$ in the definition of the total variation, we can define $\|Du\|(\Omega)$.
For an arbitrary set $A\subset X$, we define
\[
\|Du\|(A)=\inf\{\|Du\|(\Omega):\, A\subset\Omega,\,\Omega\subset X
\text{ is open}\}.
\]
In general, if $A\subset X$ is an arbitrary set, we understand the statement $\Vert Du\Vert(A)<\infty$
to mean that there exists some open set $\Om\supset A$ such that $u\in L^1_{\loc}(\Om)$ and $\Vert Du\Vert(\Om)<\infty$.
If $\Vert Du\Vert(\Omega)<\infty$, $\|Du\|(\cdot)$ is a finite Radon measure on $\Omega$ by \cite[Theorem 3.4]{M}.
A $\mu$-measurable set $E\subset X$ is said to be of finite perimeter if $\|D\ch_E\|(X)<\infty$, where $\ch_E$ is the characteristic function of $E$.
The perimeter of $E$ in $\Omega$ is also denoted by
\[
P(E,\Omega):=\|D\ch_E\|(\Omega).
\]

We have the following coarea formula from \cite[Proposition 4.2]{M}: if $\Omega\subset X$ is an open set and $u\in L^1_{\loc}(\Omega)$, then
\begin{equation}\label{eq:coarea}
\|Du\|(\Omega)=\int_{-\infty}^{\infty}P(\{u>t\},\Omega)\,dt.
\end{equation}

We will assume throughout the paper that $X$ supports a $(1,1)$-Poincar\'e inequality,
meaning that there exist constants $C_P\ge 1$ and $\lambda \ge 1$ such that for every
ball $B(x,r)$, every $u\in L^1_{\loc}(X)$,
and every upper gradient $g$ of $u$,
we have 
\[
\vint{B(x,r)}|u-u_{B(x,r)}|\, d\mu 
\le C_P r\vint{B(x,\lambda r)}g\,d\mu,
\]
where 
\[
u_{B(x,r)}:=\vint{B(x,r)}u\,d\mu :=\frac 1{\mu(B(x,r))}\int_{B(x,r)}u\,d\mu.
\]

\section{Preliminary results}\label{sec:preliminary results}

In this section we consider certain preliminary results that we will need in proving the main
theorems.
We start with the following simple result concerning Newton-Sobolev functions with zero
boundary values.

\begin{lemma}\label{lem:zero boundary values}
Let $\Om\subset X$ be an open set, let $u\in N^{1,1}(\Om)$ with $-1\le u\le 1$, and
let $\eta\in N^{1,1}_0(\Om)$ with $0\le\eta\le 1$. Then $\eta u\in N_0^{1,1}(\Om)$ with a
$1$-weak upper gradient $\eta g_u+|u|g_{\eta}$ (in $X$).
\end{lemma}
Here $g_{u}$ and $g_{\eta}$ are the minimal $1$-weak upper gradients of $u$ and $\eta$
(in $\Om$ and $X$, respectively).
By \cite[Corollary 2.21]{BB} we know that if $v\in N^{1,1}(X)$, then
\begin{equation}\label{eq:upper gradient in constant set}
g_v=0\ \ \textrm{in}\ \ \{v=0\}
\end{equation}
($\mu$-almost everywhere, to be precise). Thus $g_{\eta}=0$ outside $\Om$,
and so the function $\eta g_u+|u|g_{\eta}$
can be interpreted to take the value zero outside $\Om$.

\begin{proof}
By the Leibniz rule, see \cite[Theorem 2.15]{BB}, we know that $\eta u\in N^{1,1}(\Om)$ with a
$1$-weak upper gradient $\eta g_u+|u|g_\eta$ in $\Om$.
Moreover, $-\eta\le \eta u\le \eta\in N^{1,1}_0(\Om)$, and then by \cite[Lemma 2.37]{BB} we conclude
$\eta u\in N_0^{1,1}(\Om)$. Finally, by \eqref{eq:upper gradient in constant set} we know that
$\eta g_u+|u|g_{\eta}$ is a $1$-weak upper gradient of $u\eta$ in $X$.
\end{proof}

The following two lemmas describe two ways of enlarging a set without
increasing the $1$-capacity significantly.

\begin{lemma}[{\cite[Lemma 3.1]{L}}]\label{lem:covering G by a set of finite perimeter}
For any $G\subset X$ and $\eps>0$ there exists an open set $V\supset G$ with
$\capa_1(V)\le C_1(\capa_1(G)+\eps)$ and $P(V,X)\le C_1(\capa_1(G)+\eps)$,
for a constant $C_1=C_1(C_d,C_P,\lambda)\ge 1$.
\end{lemma}

\begin{proof}
See \cite[Lemma 3.1]{L}; note that there was a slight error in the
formulation, as the possibility $\capa_1(G)=0$ was not taken into account,
but this is easily corrected by adding an $\eps$-term in suitable places.
\end{proof}

\begin{lemma}\label{lem:capacity and Newtonian function}
Let $G\subset X$ and $\eps>0$. There exists an open set $V\supset G$ with $\capa_1(V)\le C_2(\capa_1(G)+\eps)$ and a
function $\eta\in N^{1,1}_0(V)$ with $0\le\eta\le 1$ on $X$, $\eta=1$ on $G$, and $\Vert \eta\Vert_{N^{1,1}(X)}\le C_2(\capa_1(G)+\eps)$,
for some constant $C_2=C_2(C_d,C_P,\lambda)\ge 1$.
\end{lemma}

\begin{proof}
By Lemma \ref{lem:covering G by a set of finite perimeter} we find an open set
$V_0\supset G$ with
\[
\capa_1(V_0)\le C_1(\capa_1(G)+\eps)\ \ \ \textrm{and}\ \ \ P(V_0,X)\le C_1(\capa_1(G)+\eps).
\]
By a suitable \emph{boxing inequality}, see \cite[Lemma 4.2]{HaKi}, we find balls $\{B(x_i,r_i)\}_{i=1}^{\infty}$ with $r_i\le 1$ covering $V_0$, and
\[
\sum_{i=1}^{\infty}\frac{\mu(B(x_i,r_i))}{r_i}\le C_B (\mu(V_0)+P(V_0,X))
\]
for some constant $C_B=C_B(C_d,C_P,\lambda)>0$.
For each $i\in\N$, take a $1/r_i$-Lipschitz function $0\le f_i\le 1$ with $f_i=1$ on $B(x_i,2r_i)$ and
$f_i=0$ on $X\setminus B(x_i,4r_i)$.
Let $f:=\sup_{i\in\N} f_i$. By \eqref{eq:upper gradient in constant set} and the fact that
the local Lipschitz constant is an upper gradient, $\ch_{B(x_i,4r_i)}/r_i$ is a $1$-weak upper
gradient of $f_i$. Hence
the minimal $1$-weak upper gradient of $f$ satisfies $g_f\le\sum_{i=1}^{\infty}\ch_{B(x_i,4r_i)}/r_i$,
see e.g. \cite[Lemma 1.28]{BB}. Then
\begin{align*}
\int_X g_{f}\,d\mu
\le\sum_{i=1}^{\infty}\frac{\mu(B(x_i,4r_i))}{r_i}
&\le C_d^2\sum_{i=1}^{\infty}\frac{\mu(B(x_i,r_i))}{r_i}\\
&\le C_d^2 C_B (\mu(V_0)+P(V_0,X))\\
&\le C_d^2 C_B (\capa_1(V_0)+P(V_0,X))\\
&\le 2C_d^2 C_B C_1(\capa_1(G)+\eps).
\end{align*}
Moreover, since $r_i\le 1$ for each $i\in\N$,
\[
\int_X f\,d\mu\le \sum_{i=1}^{\infty}\int_X f_i\,d\mu
\le \sum_{i=1}^{\infty}\frac{\mu(B(x_i,4r_i))}{r_i}\le 2C_d^2 C_B C_1(\capa_1(G)+\eps).
\]
Let $V:=\bigcup_{i=1}^{\infty}B(x_i,2r_i)$.
Since $f\ge 1$ on $V$, we get the estimate
\[
\capa_1(V)\le \Vert f\Vert_{N^{1,1}(X)}\le 4C_d^2 C_B C_1(\capa_1(G)+\eps).
\]
On the other hand, for each $i\in\N$, we can also take a $1/r_i$-Lipschitz function $0\le \eta_i\le 1$ with $\eta_i=1$ on $B(x_i,r_i)$ and
$\eta_i=0$ on $X\setminus B(x_i,2r_i)$.
Let $\eta:=\sup_{i\in\N} \eta_i$. Then $\eta=1$ on $V_0\supset G$ and $\eta=0$ on $X\setminus V$, and similarly as for
the function $f$, we can estimate $\Vert \eta\Vert_{N^{1,1}(X)}\le 4C_d C_B C_1\capa_1(G)$.
Thus we can choose $C_2=4C_d^2 C_B C_1$.
\end{proof}

The next lemma states that in the definition of the total variation, we can
consider convergence in $L^1(\Om)$
instead of convergence in  $L_{\loc}^1(\Om)$.

\begin{lemma}[{\cite[Lemma 5.5]{KLLS}}]\label{lem:L1 loc and L1 convergence}
Let $\Omega\subset X$ be an open set and let $u\in L^1_{\loc}(\Omega)$
with $\Vert Du\Vert(\Omega)<\infty$. Then there exists a sequence
$(w_i)\subset \liploc(\Omega)$ with $w_i-u\to 0$ in $L^1(\Omega)$ and 
$\int_\Omega g_{w_i}\, d\mu\to \Vert Du\Vert(\Omega)$.
\end{lemma}

Recall that $g_{w_i}$ denotes the minimal $1$-weak upper gradient of $w_i$ (in $\Om$).
Note that above, we cannot write $w_i\to u$ in $L^1(\Omega)$, since the functions $w_i,u$ are not necessarily
in the class $L^1(\Om)$.

\begin{lemma}[{\cite[Lemma 9.3]{BB-OD}}]\label{lem:quasiopen sets are measurable}
Every $1$-quasiopen set is $\mu$-measurable.
\end{lemma}
In fact, this is proved for all $1\le p<\infty$ in the above reference, but we only need the case $p=1$.

The coarea formula \eqref{eq:coarea} states that if $\Omega\subset X$ is an open set and $u\in L^1_{\loc}(\Omega)$, then
\[
\|Du\|(\Omega)=\int_{-\infty}^{\infty}P(\{u>t\},\Omega)\,dt.
\]
If $\Vert Du\Vert(\Omega)<\infty$, the above is true with $\Omega$ replaced by any Borel set $A\subset\Omega$; this is also given in \cite[Proposition 4.2]{M}. However, one can construct simple examples of non-Borel $1$-quasiopen sets, so we need to verify the coarea formula for such sets separately.
In doing this, we use the following lemma,
which states that the total variation of a $\BV$ function is absolutely continuous with respect
to the $1$-capacity.

\begin{lemma}[{\cite[Lemma 3.9]{L2}}]\label{lem:absolute cont of variation measure wrt capacity}
Let $\Omega\subset X$ be an open set, and
let $u\in L^1_{\loc}(\Omega)$ with $\Vert Du\Vert(\Omega)<\infty$. Then for every $\eps>0$ there exists $\delta>0$ such that if $A\subset \Omega$ with
$\capa_1(A)<\delta$, then $\Vert Du\Vert(A)<\eps$.
\end{lemma}

\begin{proposition}\label{prop:coarea generalization}
Let $U\subset X$ be a $1$-quasiopen set and suppose that $\Vert Du\Vert(U)<\infty$.
Then
\[
\Vert Du\Vert(U)=\int_{-\infty}^{\infty}P(\{u>t\},U)\,dt.
\]
\end{proposition}

\begin{proof}
Recall that implicit in the condition $\Vert Du\Vert(U)<\infty$ is the requirement
that there exists an open set
$\Om\supset U$ such that $u\in L^1_{\loc}(\Om)$ and $\Vert Du\Vert(\Om)<\infty$.
Since $U$ is $1$-quasiopen, we can pick open sets $G_i\subset X$ such that
$\capa_1(G_i)\to 0$ and each $U\cup G_i$ is an open set, and we can also assume that
$G_i\subset \Om$ and
$G_{i+1}\subset G_i$
for each $i\in\N$.
Then by the coarea formula \eqref{eq:coarea},
\[
\Vert Du\Vert(U\cup G_i)=\int_{-\infty}^{\infty}P(\{u>t\},U\cup G_i)\,dt.
\]
By Lemma \ref{lem:absolute cont of variation measure wrt capacity},
$\Vert Du\Vert(U\cup G_i)\to \Vert Du\Vert(U)$ as $i\to\infty$. Similarly,
$P(\{u>t\},U\cup G_i)\to P(\{u>t\},U)$ for every $t\in\R$ for which
$P(\{u>t\},U\cup G_1)<\infty$,
that is, for a.e. $t\in\R$. Then by Lebesgue's dominated convergence theorem,
with the majorant function $t\mapsto P(\{u>t\},U\cup G_1)$, we obtain
\[
\Vert Du\Vert(U)=\int_{-\infty}^{\infty}P(\{u>t\},U)\,dt.
\]
\end{proof}

\section{Main results}

In this section we state and prove our main results on the characterization
and lower semicontinuity of the total variation in $1$-quasiopen sets.

The definition of the total variation states that if $\Om\subset X$
is an open set and $u\in L^1_{\loc}(\Om)$, then
\begin{equation}\label{eq:definition of total variation repeated}
\|Du\|(\Om)=\inf\left\{\liminf_{i\to\infty}\int_\Om g_{u_i}\,d\mu:\, u_i\in \Lip_{\loc}(\Om),\, u_i\to u\textrm{ in } L^1_{\loc}(\Om)\right\},
\end{equation}
where each $g_{u_i}$ is the minimal $1$-weak upper gradient of $u_i$ in $\Om$.
Moreover, by \cite[Theorem 5.47]{BB}, for any $v\in N^{1,1}_{\loc}(\Om)$
and $\eps>0$ we can find $w\in \liploc(\Om)$ with $\Vert v-w\Vert_{N^{1,1}(\Om)}<\eps$.
Thus in the above definition we can equivalently assume that
$u_i\in N^{1,1}_{\loc}(\Om)$.

\begin{example}\label{ex:total variation}
Let $X=\R$ (unweighted), let $A:=[0,1]$, and let $u:=\ch_{A}$.
Then by definition
\[
\Vert Du\Vert(A)=\inf\{\Vert Du\Vert(\Om):\ \Om\supset A,\ \Om\textrm{ open}\}=2.
\]
On the other hand, the constant sequence $u_i:=1$ in $A$, $i\in\N$, converges to $u$ in $L^1(A)$ and
has upper gradients $g_{u_i}=0$ in $A$. This demonstrates that we cannot obtain $\Vert Du\Vert(A)$ simply by 
writing \eqref{eq:definition of total variation repeated} with $\Om$ replaced by $A$.
If we did define $\Vert Du\Vert(D)$ in this way for all ($\mu$-measurable) sets $D\subset \R$, then
we would obtain $\Vert Du\Vert(\R)=2$, $\Vert Du\Vert(A)=0$, and $\Vert Du\Vert(\R\setminus A)=0$,
so that $\Vert Du\Vert$ would not be a measure.
\end{example}

However, for $1$-quasiopen sets we have the following.

\begin{theorem}\label{thm:characterization of total variational}
Let $U\subset X$ be a $1$-quasiopen set. If $\Vert Du\Vert(U)<\infty$, then
\[
\Vert Du\Vert(U)=\inf \left\{\liminf_{i\to\infty}\int_{U}g_{u_i}\,d\mu,\,
u_i\in N_{\loc}^{1,1}(U),\, u_i\to u\textrm{ in }L^1_{\loc}(U)\right\},
\]
where each $g_{u_i}$ is the minimal $1$-weak upper gradient of $u_i$ in $U$.
\end{theorem}

Note that the condition $u_i\to u$ in $L_{\loc}^1(U)$ means, explicitly, that
for every $x\in U$ there exists $r>0$ such that $u_i\to u$ in $L^1(B(x,r)\cap U)$.
In order for the formulation of the theorem to make sense, we need $U$ to be $\mu$-measurable, which is guaranteed by
Lemma \ref{lem:quasiopen sets are measurable}.

First we prove the following weaker version.

\begin{proposition}\label{prop:characterization of total variational for bounded functions}
Let $U\subset X$ be a $1$-quasiopen set. If $\Om\supset U$ is an open set and
$-1\le u\le 1$ is a $\mu$-measurable function on $\Om$
with $\Vert Du\Vert(\Om)<\infty$, then
\[
\Vert Du\Vert(U)=\inf \left\{\liminf_{i\to\infty}\int_{U}g_{u_i}\,d\mu,\,
u_i\in N_{\loc}^{1,1}(U),\, u_i\to u\textrm{ in }L^1_{\loc}(U)\right\},
\]
where each $g_{u_i}$ is the minimal $1$-weak upper gradient of $u_i$ in $U$.
\end{proposition}

\begin{proof}
Denote the infimum in the statement of the theorem by $a(u,U)$.
Clearly $a(u,U)\le \Vert Du\Vert(U)$, so we only need to prove that $\Vert Du\Vert(U)\le a(u,U)$.
We can assume that $a(u,U)<\infty$.
First assume also that $u\in\BV(X)$ with $-1\le u\le 1$. Fix $\eps>0$.
By Lemma \ref{lem:absolute cont of variation measure wrt capacity}, there exists $\delta\in (0,\eps)$
such that if $A\subset X$ with $\capa_1(A)<\delta$, then $\Vert Du\Vert(A)<\eps$.
Take a sequence $(u_i)\subset N_{\loc}^{1,1}(U)$ with $u_i\to u$ in $L^1_{\loc}(U)$ and
\[
\liminf_{i\to\infty}\int_U g_{u_i}\,d\mu\le a(u,U)+\eps.
\]
By truncating, we can also assume that $-1\le u_i\le 1$.
Then take an open set $G\subset X$ such that $\capa_1(G)<\delta/C_2$ and $U\cup G$ is open.
By Lemma \ref{lem:capacity and Newtonian function} we find
an open set $V\supset G$
with $\capa_1(V)<\delta$ and
a function $\eta\in N_0^{1,1}(V)$ with $0\le\eta\le 1$, $\eta=1$ on $G$, and
$\Vert \eta\Vert_{N^{1,1}(X)}<\delta$.

By the definition of the total variation, we find a
sequence $(v_i)\subset N_{\loc}^{1,1}(V)$ with $v_i\to u$ in $L^1_{\loc}(V)$ and
\[
\Vert Du\Vert(V)=\lim_{i\to\infty}\int_{V} g_{v_i}\,d\mu.
\]
We can again assume that $-1\le v_i\le 1$, and then in fact $v_i\to u$ in $L^1(V)$.
Define
\[
w_i:=(1-\eta) u_i+\eta v_i,\quad i\in\N.
\]
By the Leibniz rule, see \cite[Theorem 2.15]{BB}, $(1-\eta)u_i$ has a $1$-weak upper gradient
\[
(1-\eta)g_{u_i}+|u_i|g_{\eta}
\]
in $U$.
By Lemma \ref{lem:zero boundary values},
\[
\eta v_i\in N_0^{1,1}(V)\subset N^{1,1}(X)\subset N^{1,1}(U)
\]
with a $1$-weak upper gradient $\eta g_{v_i}+|v_i|g_{\eta}$ (in $X$, and thus in $U$).
In total, $w_i$ has a $1$-weak upper gradient
\[
g_i:=(1-\eta)g_{u_i}+\eta g_{v_i}+2g_{\eta}
\]
in $U$. Next we show that in fact, $g_i$ is a $1$-weak upper gradient of $w_i$ in $U\cup G$;
note that
while $g_{u_i}$ is only defined on $U$, $(1-\eta)g_{u_i}$ is defined in a natural way on $U\cup G$,
and similarly for the term $\eta g_{v_i}$.

Since $U$ is a $1$-quasiopen set, it is also \emph{$1$-path open},
meaning that for $1$-a.e. curve $\gamma$,
the set $\gamma^{-1}(U)$ is a relatively open subset of $[0,\ell_{\gamma}]$, see \cite[Remark 3.5]{S2}.
Fix such a curve $\gamma$ in $U\cup G$, and assume also that the upper gradient inequality holds for the
pair $(w_i,g_i)$ on any subcurve of $\gamma$ in $U$, and for the
pair $(v_i,g_{v_i})$ on any subcurve of $\gamma$ in $G$; by \cite[Lemma 1.34(c)]{BB} this is true for $1$-a.e. curve.

Now $[0,\ell_{\gamma}]$ is a compact set that is covered by the two relatively
open sets $\gamma^{-1}(U)$ and $\gamma^{-1}(G)$.
By the Lebesgue number lemma,
there exists a number $\beta>0$ such that every subinterval of $[0,\ell_{\gamma}]$ with
length at most $\beta$ is contained either in $\gamma^{-1}(U)$ or in
$\gamma^{-1}(G)$.
Choose $m\in\N$ such that $\ell_{\gamma}/m\le \delta$ and consider the
subintervals $I_j:=[j\ell_{\gamma}/m,(j+1)\ell_{\gamma}/m]$, $j=0,\ldots,m-1$.
If $I_j\subset \gamma^{-1}(U)$, then by our assumptions on $\gamma$,
\[
|w_i(j\ell_{\gamma}/m)-w_i((j+1)\ell_{\gamma}/m)|\le \int_{j\ell_{\gamma}/m}^{(j+1)\ell_{\gamma}/m}g_i(\gamma(s))\,ds.
\]
Otherwise $I_j\subset \gamma^{-1}(G)$. Recall that
$\eta=1$ on $G$. Then by our assumptions on $\gamma$,
\begin{align*}
|w_i(j\ell_{\gamma}/m)-w_i((j+1)\ell_{\gamma}/m)|
&=|v_i(j\ell_{\gamma}/m)-v_i((j+1)\ell_{\gamma}/m)|\\
&\le \int_{j\ell_{\gamma}/m}^{(j+1)\ell_{\gamma}/m}g_{v_i}(\gamma(s))\,ds\\
&=\int_{j\ell_{\gamma}/m}^{(j+1)\ell_{\gamma}/m}g_i(\gamma(s))\,ds.
\end{align*}
Adding up the inequalities for $j=1,\ldots,m-1$, we conclude that the
upper gradient inequality holds for the
pair $(w_i,g_i)$ on the curve $\gamma$, that is,
\[
|w_i(0)-w_i(\ell_{\gamma})|\le \int_0^{\ell_{\gamma}}g_i(\gamma(s))\,ds.
\]
Thus
$g_i$ is a $1$-weak upper gradient of
$w_i$ in the open set $U\cup G$.
Next we show that $w_i\to u$ in $L_{\loc}^1(U\cup G)$.
Let $x\in U$. Since $u_i\to u$ in $L^1_{\loc}(U)$, there is some $r>0$ such that $u_i\to u$
in $L^1(B(x,r)\cap U)$. Moreover, $v_i\to u$ in $L^1(V)$, and by making $r$ smaller, if necessary,
$B(x,r)\subset U\cup G$.
Then
\begin{align*}
\int_{B(x,r)}|w_i-u|\,d\mu
&\le\int_{B(x,r)}|(1-\eta)(u_i-u)|\,d\mu +
\int_{B(x,r)}|\eta (v_i-u)|\,d\mu\\
&\le \int_{B(x,r)\cap U\setminus G}|u_i-u|\,d\mu+
\int_{B(x,r)\cap V}|v_i-u|\,d\mu\\
&\to 0\qquad \textrm{as }i\to\infty.
\end{align*}
On the other hand, if $x\in G$, then for some $r>0$, $B(x,r)\subset G$. Then
\[
\int_{B(x,r)}|w_i-u|\,d\mu= \int_{B(x,r)}|v_i-u|\,d\mu\to 0.
\]
We conclude that $w_i\to u$ in $L_{\loc}^1(U\cup G)$. Now by the definition of the total variation
(recall \eqref{eq:definition of total variation repeated} and the discussion after it)
\begin{align*}
\Vert Du\Vert(U\cup G)
&\le \liminf_{i\to\infty}\int_{U\cup G}g_i\,d\mu\\
&\le \liminf_{i\to\infty}\int_{U}g_{u_i}\,d\mu+\limsup_{i\to\infty}\int_{V}g_{v_i}\,d\mu+2\limsup_{i\to\infty}\int_X g_{\eta}\,d\mu\\
&\le a(u,U)+\eps+\Vert Du\Vert(V)+2\int_X g_{\eta}\,d\mu\\
&< a(u,U)+4\eps;
\end{align*}
recall that $\Vert Du\Vert(V)<\eps$ since $\capa_1(V)<\delta$, and that
$\Vert \eta\Vert_{N^{1,1}(X)}<\delta<\eps$.
In conclusion,
\[
\Vert Du\Vert(U)\le \Vert Du\Vert(U\cup G)\le a(u,U)+4\eps.
\]
Letting $\eps\to 0$,
the proof is complete in the case $u\in\BV(X)$, $-1\le u\le 1$.

Now we drop the assumption $u\in\BV(X)$. By assumption, we have $-1\le u\le 1$ on the open set $\Om\supset U$,
with $\Vert Du\Vert(\Om)<\infty$.
Take open sets $\Omega_1\Subset \Omega_2\Subset\ldots \Subset \Omega$ with $\Omega=\bigcup_{j=1}^{\infty}\Omega_j$,
and cutoff functions $\eta_j\in \Lip_c(X)$ with $0\le \eta_j\le 1$, $\eta_j=1$ on $\Omega_j$, and $\eta_j=0$ on $X\setminus \Omega_{j+1}$.
Fix $j\in\N$.
It is easy to check that $u \eta_j\in\BV(X)$ for each $j\in\N$.
Since $\Omega_j\cap U$ is a $1$-quasiopen set\footnote{Quasiopen sets do not form a topology, see \cite[Remark 9.1]{BB-OD}, but it is easy to see that the intersection of a $1$-quasiopen set and an open set is $1$-quasiopen.},
we get by the first part of the proof that
\begin{align*}
\Vert Du\Vert(\Om_j\cap U)
&=\Vert D(u\eta_j)\Vert(\Om_j\cap U)\\
&= a(u\eta_j,\Om_j\cap U)
= a(u,\Om_j\cap U)\le a(u,U).
\end{align*}
Letting $j\to\infty$ concludes the proof.
\end{proof}

Before proving Theorem \ref{thm:characterization of total variational},
we prove our second main result, which states that the total variation of $\BV$ functions is lower semicontinuous with
respect to $L^1$-convergence in $1$-quasiopen sets.
In fact, we will use this to prove Theorem \ref{thm:characterization of total variational}.

\begin{theorem}\label{thm:lower semic in quasiopen sets}
Let $U\subset X$ be a $1$-quasiopen set.
If $\Vert Du\Vert(U)<\infty$ and $u_i\to u$ in $L^1_{\loc}(U)$, then
\[
\Vert Du\Vert(U)\le \liminf_{i\to\infty}\Vert Du_i\Vert(U).
\]
\end{theorem}

\begin{proof}
First assume that $E,E_i\subset X$, $i\in\N$, are $\mu$-measurable sets with
$P(E,U),P(E_i,U)<\infty$ and
$\ch_{E_i}\to \ch_E$ in $L_{\loc}^1(U)$.
For each $i\in\N$, the condition $P(E_i,U)<\infty$ means that we find an open set $\Om_i\supset U$ such that
$P(E_i,\Om_i)<P(E_i,U)+1/i<\infty$.
Then by Lemma \ref{lem:L1 loc and L1 convergence}, for each $i\in\N$ we find a function
$v_i\in \liploc(\Om_i)\subset N_{\loc}^{1,1}(\Om_i)$ such that
\[
\Vert v_i-\ch_{E_i}\Vert_{L^1(\Om_i)}<1/i\quad\textrm{and}\quad\int_{\Om_i}
g_{v_i}\,d\mu<P(E_i,\Om_i)+1/i,
\]
where $g_{v_i}$ is the minimal $1$-weak upper gradient of $v_i$ in $\Om_i$.
In particular, we have $v_i\in N_{\loc}^{1,1}(U)$ with
\[
\Vert v_i-\ch_{E_i}\Vert_{L^1(U)}<1/i\quad\textrm{and}\quad\int_{U} g_{v_i}\,d\mu<P(E_i,U)+2/i,
\]
where $g_{v_i}$ is now the minimal $1$-weak upper gradient of $v_i$ in $U$, which
is of course at most the minimal $1$-weak upper gradient of $v_i$ in $\Om_i$.
Now we clearly have $v_i\to \ch_E$ in $L^1_{\loc}(U)$. Moreover, the condition $P(E,U)<\infty$
means that there exists an open set $\Om\supset U$ such that $P(E,\Om)<\infty$. Thus by Proposition
\ref{prop:characterization of total variational for bounded functions},
\[
P(E,U)\le \liminf_{i\to\infty}\int_U g_{v_i}\,d\mu\le\liminf_{i\to\infty}( P(E_i,U)+2/i)
=\liminf_{i\to\infty} P(E_i,U).
\]
Thus we have proved lower semicontinuity in the case of sets of finite perimeter.
Then consider the function $u$.
Note that it is enough to prove the lower semicontinuity for a subsequence.
We have $u_i\to u$ in $L_{\loc}^1(U)$, which means that for every $x\in U$ there
exists $r_x>0$ such that $u_i\to u$ in $L^1(B(x,r_x)\cap U)$. Consider the cover
$\{B(x,r_x)\}_{x\in U}$.
We know that the space $X$ is
separable, see e.g. \cite[Proposition 1.6]{BB},
and this property is inherited by subsets of $X$. Thus $U$ is separable,
and so it is also Lindel\"of, meaning that every open cover of $U$ has a countable subcover,
see \cite[pp. 176--177]{Kur}.
Thus there exists a countable subcover $\{B(x_j,r_j)\}_{j\in\N}$ of $U$.

Consider the ball $B(x_1,r_1)$.
We have $u_i\to u$ in $L^1(B(x_1,r_1)\cap U)$, and so
by passing to a subsequence (not relabeled), for a.e. $t\in\R$ we have $\ch_{\{u_i>t\}}\to \ch_{\{u>t\}}$ in $L^1(B(x_1,r_1)\cap U)$, see e.g. \cite[p. 188]{EvaG92}.
By a diagonal argument, we find a subsequence (not relabeled) such that for each $j\in\N$ and
a.e. $t\in\R$ we have $\ch_{\{u_i>t\}}\to \ch_{\{u>t\}}$ in $L^1(B(x_j,r_j)\cap U)$.
Since the balls $B(x_j,r_j)$ cover $U$, we conclude that for a.e. $t\in\R$,
$\ch_{\{u_i>t\}}\to \ch_{\{u>t\}}$ in $L_{\loc}^1(U)$.

We can assume that $\Vert Du_i\Vert(U)<\infty$ for all $i\in\N$, and so by
Proposition \ref{prop:coarea generalization},
$\int_{-\infty}^{\infty}P(\{u_i>t\},U)\,dt<\infty$ for all $i\in\N$,
and in particular,
for each $i\in\N$ the mapping $t\mapsto P(\{u_i>t\},U)$ is measurable,
enabling us to use Fatou's lemma.
Moreover, we are able to use the lower semicontinuity for sets of finite perimeter proved above,
because for a.e. $t\in\R$, $P(\{u>t\},U)<\infty$ and $P(\{u_i>t\},U)<\infty$ for all $i\in\N$.
Indeed, now we use Proposition \ref{prop:coarea generalization},
the lower semicontinuity for sets of finite perimeter proved above, and Fatou's lemma to obtain
\begin{align*}
\Vert Du\Vert(U)
=\int_{-\infty}^{\infty}P(\{u>t\},U)\,dt
&\le\int_{-\infty}^{\infty}\liminf_{i\to\infty}P(\{u_i>t\},U)\,dt\\
&\le\liminf_{i\to\infty}\int_{-\infty}^{\infty}P(\{u_i>t\},U)\,dt\\
&=\liminf_{i\to\infty}\Vert Du_i\Vert(U).
\end{align*}
\end{proof}

Knowing that the total variation is lower semicontinuous in a wider class of sets than just the open sets should prove useful in dealing with various minimization problems. In the upcoming work \cite{LMS} we need lower semicontinuity of the total variation in the super-level sets of a given Newton-Sobolev function $w\in N^{1,1}(X)$. Such sets are $1$-quasiopen since functions in the class $N^{1,1}(X)$ are $1$-quasicontinuous; see \cite{BBM-QP} for more on these concepts.

Finally, we give the proof of Theorem \ref{thm:characterization of total variational}.

\begin{proof}[Proof of Theorem \ref{thm:characterization of total variational}]
Suppose that $\Vert Du\Vert(U)<\infty$.
First suppose also that there exists $M>0$ and an open set $\Om\supset U$ such
that $-M\le u\le M$ on $\Om$,
and $\Vert Du\Vert(\Om)<\infty$. Again, denote by $a(u,U)$ the infimum in the statement of the theorem.
It is obvious that a function $g$ is a $1$-weak upper gradient of a function $v$
if and only if $g/M$ is a $1$-weak upper gradient of $v/M$.
Using this fact and Proposition \ref{prop:characterization of total variational for bounded functions}, we obtain
\[
\Vert Du\Vert(U)/M=\Vert D(u/M)\Vert(U)= a(u/M,U)=a(u,U)/M,
\]
so that
\[
\Vert Du\Vert(U)= a(u,U).
\]
Then suppose we only have $\Vert Du\Vert(U)<\infty$.
This means that there exists an open set $\Om\supset U$ such that $u\in L^1_{\loc}(\Om)$ and
$\Vert Du\Vert(\Om)<\infty$.
Define the truncations
\[
u_M:=\min\{M,\max\{-M,u\}\},\quad M>0,
\]
and apply Theorem \ref{thm:lower semic in quasiopen sets} and the
first part of the current proof to obtain
\begin{align*}
\Vert Du\Vert(U)
&\le \liminf_{M\to\infty}\Vert Du_M\Vert(U)\\
&= \liminf_{M\to\infty}a(u_M,U)
\le\liminf_{M\to\infty}a(u,U)=a(u,U).
\end{align*}
\end{proof}

Contrary to the case of open sets, lower semicontinuity can actually be violated in $1$-quasiopen sets if the limit function is not a $\BV$ function. Thus the requirement
$\Vert Du\Vert(U)<\infty$ in Theorem \ref{thm:lower semic in quasiopen sets} is essential.

\begin{example}\label{ex:necessity of finite variation}
Let $X=\R^2$ (unweighted). Denote the origin by $0$, and let $U:=\{0\}$.
The set $B(0,r)$ is open for all $r>0$,
and it is easy to check that $\capa_1(B(0,r))\le 3\pi r$ for $0<r\le 1$.
Thus $U$ is a $1$-quasiopen set. Let
\[
E:=\bigcup_{i=1}^{\infty}\{(x_1,x_2)\in\R^2:\,(2i+1)^{-1}< x_1< (2i)^{-1}\}.
\]
It is well known that $P(E,B(0,r))=\mathcal H^1(\partial^*E\cap B(0,r))$, where $\mathcal H^1$
is the $1$-dimensional Hausdorff measure, see e.g. \cite[Theorem 3.59]{AFP}. 
Thus clearly $P(E,B(0,r))=\infty$ for all $r>0$, and so
$P(E,U)=\infty$.
Next, let
\[
E_k:=\bigcup_{i=1}^k\{(x_1,x_2)\in\R^2:\,(2i+1)^{-1}< x_1< (2i)^{-1}\},\quad k\in\N.
\]
Then $\ch_{E_k}\to \ch_{E}$ even in $L_{\loc}^1(\R^2)$ (and obviously in $L_{\loc}^1(U)$).
On the other hand, for any $k\in\N$ and $r<(2k+1)^{-1}$,
\[
P(E_k,U)\le P(E_k, B(0,r))=0,
\]
since $E_k$ does not intersect the open set $ B(0,r)$.
Thus
\[
P(E,U)>\lim_{k\to\infty}P(E_k,U),
\]
that is, lower semicontinuity is violated.

Similarly we see that without the assumption
$\Vert Du\Vert(U)<\infty$, Theorem \ref{thm:characterization of total variational} fails with the choice
$u=\ch_E$, as the left-hand side is $\infty$ but the right-hand side is zero.
\end{example}

It would be interesting to know if the conclusions of Theorem \ref{thm:characterization of total variational} and Theorem \ref{thm:lower semic in quasiopen sets}
actually
characterize $1$-quasiopen sets.

\begin{openproblem}
Let $U\subset X$ be a $\mu$-measurable set such that the conclusion of Theorem
\ref{thm:characterization of total variational} or the conclusion of Theorem
\ref{thm:lower semic in quasiopen sets} holds.
Is $U$ then a $1$-quasiopen set?
\end{openproblem}

To conclude this section, we apply Lemma \ref{lem:capacity and Newtonian function} to prove a somewhat different but quite natural characterization of $1$-quasiopen sets, given in Proposition \ref{prop:quasiopen sets characterization} below.
For other characterizations of quasiopen sets, see \cite{BBM-QP}.
First we take note of the following facts.
By \cite[Theorem 4.3, Theorem 5.1]{HaKi} we know that for any $A\subset X$,
\begin{equation}\label{eq:null sets of Hausdorff measure and capacity}
\capa_1(A)=0\quad\ \textrm{if and only if}\quad\ \mathcal H(A)=0.
\end{equation}

The following proposition follows from
\cite[Corollary 4.2]{L2} (which is originally based on \cite[Theorem 1.1]{LaSh}).

\begin{proposition}\label{prop:quasisemicontinuity}
Let $u\in\BV_{\loc}(X)$ and let $\eps>0$. Then there exists an open set $G\subset X$ with $\capa_1(G)<\eps$
such that $u^{\wedge}|_{X\setminus G}$ is real-valued lower semicontinuous and
$u^{\vee}|_{X\setminus G}$ is
real-valued upper semicontinuous.
\end{proposition}

Moreover, $1$-quasiopen sets can be perturbed in the following way.

\begin{lemma}\label{lem:stability of quasiopen sets}
Let $U\subset X$ be a $1$-quasiopen set and let $A\subset X$ be $\mathcal H$-negligible.
Then $U\setminus A$ and $U\cup A$ are $1$-quasiopen sets.
\end{lemma}
\begin{proof}
Let $\eps>0$. Take an open set $G\subset X$ such that $\capa_1(G)<\eps$ and $U\cup G$ is an open set.
By \eqref{eq:null sets of Hausdorff measure and capacity} we know that $\capa_1(A)=0$, and since $\capa_1$
is an outer capacity, we find an open set $V\supset A$ such that $\capa_1(G)+\capa_1(V)<\eps$.
Now $(U\setminus A)\cup (G\cup V)=(U\cup G)\cup V$ is an open set with $\capa_1(G\cup V)<\eps$, so that
$U\setminus A$ is a $1$-quasiopen set. Similarly, $(U\cup A)\cup (G\cup V)=(U\cup G)\cup V$
is an open set, so that $U\cup A$ is  also a $1$-quasiopen set.
\end{proof}

In the following, $\Delta$ denotes the symmetric difference.

\begin{proposition}\label{prop:quasiopen sets characterization}
Let $U\subset X$. The following are equivalent:
\begin{enumerate}[{(1)}]
\item $U$ is $1$-quasiopen.
\item There exists  $u\in N_{\loc}^{1,1}(X)$ with $\mathcal H(\{u>0\}\Delta U)=0$.
\item There exists  $u\in\BV_{\loc}(X)$ with $\mathcal H(\{u^{\wedge}>0\}\Delta U)=0$.
\end{enumerate}
\end{proposition}

\begin{proof}
\hfill
\begin{itemize}
\item $(1)\implies (2)$: Take a sequence of open sets $G_i\subset X$ such that $\capa_1(G_i)<2^{-i}$
and each $U\cup G_i$ is an open set. Then define
\[
v_i(x):=2^{-i}\min\{1,\dist(x,X\setminus (U\cup G_i))\},\quad x\in X,\ \ i\in\N.
\]
By
Lemma \ref{lem:capacity and Newtonian function} there exist sets $V_i\supset G_i$ and
functions $\eta_i\in N_0^{1,1}(V_i)$ such that $0\le\eta_i\le 1$, $\eta_i=1$ on $G_i$, and
$\Vert \eta_i\Vert_{N^{1,1}(X)}\le 2^{-i}C_2$.
Let $u_i:=(v_i-\eta_i)_+$.
Then $0\le u_i\le 1$, $u_i=v_i>0$ on $U\setminus V_i$, and $u_i=0$
on $G_i$. Since also $u_i\le v_i=0$ on $X\setminus (U\cup G_i)$, in total $u_i=0$ on $X\setminus U$.
For any bounded open set $\Om\subset X$,
we have
\begin{align*}
\Vert u_i\Vert_{N^{1,1}(\Om)}
&\le \Vert v_i\Vert_{N^{1,1}(\Om)}+\Vert \eta_i\Vert_{N^{1,1}(\Om)}\\
&\le \Vert v_i\Vert_{L^1(\Om)}+\int_\Om g_{v_i}\,d\mu+\Vert\eta_i\Vert_{N^{1,1}(X)}\\
&\le 2^{-i}\mu(\Om)+2^{-i}\mu(\Om)+2^{-i}C_2.
\end{align*}
Let $u:=\sup_{i\in\N}u_i$.
Then $\sup_{i\in\N}g_{u_i}$ is a $1$-weak upper gradient of $u$, see \cite[Lemma 1.52]{BB}.
Thus $\Vert u\Vert_{N^{1,1}(\Om)}<\infty$, and so $u\in N_{\loc}^{1,1}(X)$. Moreover, $u>0$ on $U\setminus \bigcap_{i=1}^{\infty}V_i$,
where $\capa_1(\bigcap_{i=1}^{\infty}V_i)=0$ and thus $\mathcal H(\bigcap_{i=1}^{\infty}V_i)=0$ by
\eqref{eq:null sets of Hausdorff measure and capacity}.
On the other hand, $u=0$ on $X\setminus U$.
Thus $\mathcal H(\{u>0\}\Delta U)=0$.

\item $(2)\implies (3)$: Take $u\in N_{\loc}^{1,1}(X)$ with
$\mathcal H(\{u>0\}\Delta U)=0$.
We know that $u$
has a Lebesgue point at $\mathcal H$-a.e. $x\in X$, see \cite[Theorem 4.1, Remark 4.2]{KKST}
and \eqref{eq:null sets of Hausdorff measure and capacity}. Thus $u(x)=u^{\wedge}(x)$
at $\mathcal H$-a.e. $x\in X$, and so $\mathcal H(\{u^{\wedge}>0\}\Delta U)=0$.
Furthermore, $N_{\loc}^{1,1}(X)\subset \BV_{\loc}(X)$ by the discussion after
\eqref{eq:definition of total variation repeated}. Thus $u\in\BV_{\loc}(X)$. 

\item $(3)\implies (1)$: Take  $u\in\BV_{\loc}(X)$ with $\mathcal H(\{u^{\wedge}>0\}\Delta U)=0$. By
Proposition \ref{prop:quasisemicontinuity}, there exist open sets $G_i\subset X$ such that
$\capa_1(G_i)\to 0$
and for each $i\in\N$, $u^{\wedge}|_{X\setminus G_i}$ is a lower semicontinuous function. Hence
the set $\{u^{\wedge}>0\}$ is open in the subspace topology of $X\setminus G_i$, and so
the sets $\{u^{\wedge}>0\}\cup G_i$ are open (in $X$). We conclude that $\{u^{\wedge}>0\}$ is a $1$-quasiopen
set, and then by Lemma \ref{lem:stability of quasiopen sets}, $U$ is also $1$-quasiopen.
\end{itemize}
\end{proof}

\section{Uniform absolute continuity}

In this section we use the lower semicontinuity result proved in the previous section to
show that the variation measures of a sequence of $\BV$ functions
converging in the strict sense are uniformly absolutely continuous with respect to the
$1$-capacity $\capa_1$.

First recall the following definition.
Given a $\mu$-measurable set $H\subset X$, a sequence of functions $(g_i)\subset L^1(H)$
is said to be \emph{uniformly integrable} if the following two conditions are satisfied.
First, for every $\eps>0$ there exists a $\mu$-measurable set $D\subset H$ with $\mu(D)<\infty$
such that 
\[
\int_{H\setminus D}g_i\,d\mu<\eps \quad\textrm{for all }i\in\N.
\]
Second, for every $\eps>0$ there exists $\delta>0$ such that if $A\subset H$ is a $\mu$-measurable set with $\mu(A)<\delta$, then
\[
\int_{A}g_i\,d\mu<\eps \quad\textrm{for all }i\in\N.
\]

The second condition can be called the uniform absolute continuity of the
measures $g_i\mu$ with
respect to $\mu$.
The variation measure of a $\BV$ function is usually not absolutely continuous
with respect to $\mu$, but according to Lemma
\ref{lem:absolute cont of variation measure wrt capacity}, it is absolutely
continuous with respect to the $1$-capacity.
Thus we can analogously talk about the variation measures of a sequence of $\BV$ functions being uniformly absolutely continuous with respect to the $1$-capacity.

Before stating our main theorem, we gather a few preliminary results. For these,
we will also need the concept of $\BV$-capacity, which is defined for a set $A\subset X$
by
\[
\capa_{\BV}(A):=\inf \Vert u\Vert_{\BV(X)},
\]
where the infimum is taken over all $u\in\BV(X)$ such that $u\ge 1$ in a neighborhood of $A$.
By \cite[Theorem 3.4]{HaKi} we know that if $A_1\subset A_2\subset \ldots\subset X$, then
\begin{equation}\label{eq:continuity of BVcap}
\capa_{\BV}\left(\bigcup_{i=1}^{\infty}A_i\right)=\lim_{i\to\infty} \capa_{\BV}(A_i).
\end{equation}
On the other hand, by \cite[Theorem 4.3]{HaKi} there is a constant
$C_{\textrm{cap}}(C_d,C_P,\lambda)\ge 1$ such that for any
$A\subset X$,
\begin{equation}\label{eq:Newtonian and BV capacities are comparable}
\capa_{\BV}(A)\le \capa_1(A)\le C_{\textrm{cap}}\capa_{\BV}(A).
\end{equation}
Thus the $1$-capacity and the $\BV$-capacity can often be used interchangeably, but the $\BV$-capacity
has the advantage that it is continuous with respect to increasing sequences of sets.

\begin{lemma}\label{lem:complete metric space}
Let $\Omega\subset X$ be an arbitrary set.
The space of sets $A\subset \Omega$ with $\capa_1(A)<\infty$, equipped with the metric
\[
\capa_1(A_1\Delta A_2),\quad A_1,A_2\subset \Omega,
\]
is a complete metric space if we identify sets $A_1,A_2\subset \Omega$ with
$\capa_1(A_1\Delta A_2=0$.
\end{lemma}
\begin{proof}
We know that $\capa_1$ is an outer measure, see e.g. \cite[Theorem 6.7]{BB}, and thus
it is straightforward to check that $\capa_1(\cdot\Delta\cdot)$ is indeed a metric.
In particular, note that if $A_1,A_2\subset \Om$ with $\capa_1(A_1),\capa_1(A_2)<\infty$, then
$\capa_1(A_1\Delta A_2)\le \capa_1(A_1\cup A_2)<\infty$, so the distance is always finite.
To verify completeness, let $\{A_i\}_{i\in\N}$ be a Cauchy sequence. We can pick a subsequence $\{A_{i_j}\}_{j\in\N}$ such that $\capa_1(A_{i_j}\Delta A_{i_{j+1}})<2^{-j}$ for all 
$j\in\N$. It follows that $\capa_1(A_{i_j}\Delta A_{i_{l}})<2^{-j+1}$ for all $l> j$.
Let
\[
A:=\bigcap_{k=1}^{\infty}\bigcup_{l=k}^{\infty}A_{i_l},
\]
so that $A\subset \Omega$.
For a fixed $j\in\N$, we now have
\begin{align*}
A_{i_j}\setminus \bigcap_{k=1}^{\infty}\bigcup_{l=k}^{\infty}A_{i_l}
= A_{i_j}\cap \bigcup_{k=1}^{\infty}\left(X\setminus \bigcup_{l=k}^{\infty}A_{i_l}\right)
&= \bigcup_{k=1}^{\infty}
\left(A_{i_j}\cap \left(X\setminus \bigcup_{l=k}^{\infty}A_{i_l}\right)\right)\\
&= \bigcup_{k=1}^{\infty}
\left(A_{i_j}\setminus \bigcup_{l=k}^{\infty}A_{i_l}\right).
\end{align*}
Thus by \eqref{eq:continuity of BVcap} and \eqref{eq:Newtonian and BV capacities are comparable},
\begin{align*}
\capa_1(A_{i_j}\setminus A)
&=\capa_1\left(\bigcup_{k=1}^{\infty}\left(A_{i_j}\setminus
\bigcup_{l=k}^{\infty}A_{i_l}\right)\right)\\
&\le C_{\textrm{cap}}\capa_{\BV}\left(\bigcup_{k=1}^{\infty}\left(A_{i_j}\setminus
\bigcup_{l=k}^{\infty}A_{i_l}\right)\right)\\
&=C_{\textrm{cap}}\lim_{k\to\infty}\capa_{\BV}\left(A_{i_j}\setminus \bigcup_{l=k}^{\infty}A_{i_l}\right)\\
&\le C_{\textrm{cap}}\lim_{k\to\infty}\capa_{1}\left(A_{i_j}\setminus \bigcup_{l=k}^{\infty}A_{i_l}\right)\\
&\le C_{\textrm{cap}}\lim_{k\to\infty}\capa_1\left(A_{i_j}\setminus A_{i_k}\right)\\
&\le C_{\textrm{cap}}\lim_{k\to\infty}2^{-j+1}\\
&=2^{-j+1}C_{\textrm{cap}}
\to 0\quad\textrm{as }j\to \infty.
\end{align*}
Conversely,
\begin{align*}
\capa_1(A\setminus A_{i_j})
\le \capa_1\left(\bigcup_{l=j}^{\infty}A_{i_l}\setminus A_{i_j}\right)
&=\capa_1\left(\bigcup_{l=j}^{\infty}(A_{i_{l+1}}\setminus A_{i_l})\right)\\
&\le \sum_{l=j}^{\infty}\capa_1(A_{i_{l+1}}\setminus A_{i_l})\\
&\le \sum_{l=j}^{\infty} 2^{-l}\\
&=2^{-j+1}\to 0\quad\textrm{as }j\to\infty.
\end{align*}
Thus $\capa_1(A_{i_j}\Delta A)\to 0$ as $j\to\infty$, and since $\{A_i\}_{i\in\N}$ is a Cauchy sequence, we have $\capa_1(A_i\Delta A)\to 0$ as $i\to\infty$. It is also clear that $\capa_1(A)<\infty$.
\end{proof}

The following proposition,
which follows from Proposition \ref{prop:quasisemicontinuity},
provides many $1$-quasiopen sets in which the lower semicontinuity
result of the previous section can be applied; recall the definitions of the measure theoretic interior $I_E$ and the measure theoretic exterior $O_E$ from
\eqref{eq:definition of measure theoretic interior} and \eqref{eq:definition of measure theoretic exterior}.

\begin{proposition}[{\cite[Proposition 4.2]{L3}}]\label{prop:set of finite perimeter is quasiopen}
Let $\Omega\subset X$ be an open set and let $E\subset X$ be a $\mu$-measurable set with
$P(E,\Omega)<\infty$. Then the sets $I_E\cap\Omega$ and $O_E\cap\Omega$ are $1$-quasiopen.
\end{proposition}

The $1$-capacity has the following useful rigidity property.

\begin{lemma}\label{lem:capacity of measure theoretic closure}
For any $A\subset X$, we have $\capa_1(I_A\cup \partial^*A)\le \capa_1(A)$.
\end{lemma}
\begin{proof}
This follows by combining \cite[Lemma 3.1]{L3} and \cite[Proposition 3.8]{L3}.
\end{proof}

Now we give the main result of this section. The proof is partially based on Baire's category theorem, similarly to the proof of the Vitali-Hahn-Saks theorem concerning uniformly integrable sequences of functions, see e.g. \cite[Theorem 1.30]{AFP}.

\begin{theorem}\label{thm:uniform absolute continuity}
Let $\Omega\subset X$ be an open set,
and suppose that $u_i\to u$ in $L_{\loc}^1(\Omega)$ and $\Vert Du_i\Vert(\Omega)\to \Vert Du\Vert(\Omega)$, with
$\Vert Du\Vert(\Omega)<\infty$ and $\Vert Du_i\Vert(\Omega) <\infty$ for all $i\in\N$.
Then for every $\eps>0$ there exists $\delta>0$ such that if $A\subset \Omega$
with $\capa_1(A)<\delta$, then $\Vert Du_i\Vert(A)<\eps$ for all $i\in\N$.
\end{theorem}

\begin{proof}
Fix $\eps>0$. By Lemma \ref{lem:absolute cont of variation measure wrt capacity} there
exists $\alpha>0$ such that if $D\subset \Omega$ with $\capa_1(D)<C_1\alpha$, then
$\Vert Du\Vert(D)<\eps/2$. Fix $A\subset\Omega$ with $\capa_1(A)<\alpha$.
By Lemma \ref{lem:covering G by a set of finite perimeter} we find an open set $V\supset A$
with $\capa_1(V)< C_1\alpha$ and $P(V,X)< C_1\alpha$.
By Lemma \ref{lem:capacity of measure theoretic closure}, also
$\capa_1(I_V\cup \partial^*V)< C_1\alpha$.
By Proposition \ref{prop:set of finite perimeter is quasiopen}, $\Omega\cap O_V$ is a
$1$-quasiopen set, and thus by the lower semicontinuity Theorem \ref{thm:lower semic in quasiopen sets} we get
\[
\Vert Du\Vert (\Omega\cap O_V)\le
\liminf_{i\to\infty}\Vert Du_i\Vert(\Omega\cap O_V).
\]
Since also $\Vert Du_i\Vert(\Omega)\to \Vert Du\Vert(\Omega)$, we have
\[
\Vert Du\Vert (\Omega\setminus O_V)\ge\limsup_{i\to\infty}\Vert Du_i\Vert(\Omega\setminus O_V),
\]
that is,
\[
\Vert Du\Vert (\Omega\cap (I_V\cup \partial^*V))\ge\limsup_{i\to\infty}
\Vert Du_i\Vert(\Omega\cap (I_V\cup \partial^*V)).
\]
But since $\capa_1(I_V\cup \partial^*V)< C_1\alpha$, we get
\[
\limsup_{i\to\infty}\Vert Du_i\Vert(\Omega\cap (I_V\cup \partial^*V))<\eps/2.
\]
Moreover, $A\subset  \Om\cap V\subset \Om\cap(I_V\cup \partial^*V)$,
since $V$ is open.
In conclusion,
\begin{equation}\label{eq:absolute continuity at limit}
A\subset \Omega\  \textrm{and}\ \capa_1(A)<\alpha\quad\textrm{imply}\quad \limsup_{i\to\infty}\Vert Du_i\Vert(A)<\eps/2.
\end{equation}
Consider the metric space defined in Lemma \ref{lem:complete metric space}.
Define the sets
\[
\mathcal A_k:=\{D\subset \Omega:\,\, \capa_1(D)<\infty\ \, \textrm{and}\ \, \sup_{i\ge k}\Vert Du_i\Vert(D)\le\eps/2\},\quad k\in\N.
\]
We show that these sets are closed. Fix $k\in\N$ and then
fix $i\ge k$. Let $D\subset X$ with $\capa_1(D)<\infty$. If  $D_n\in \mathcal A_k$, $n\in\N$,
is a sequence with
$\capa_1(D_n\Delta D)\to 0$, then since
$\Vert Du_i\Vert$ is absolutely continuous
with respect to $\capa_1$, we have
\[
\Vert Du_i\Vert(D)\le \liminf_{n\to\infty}(\Vert Du_i\Vert(D\setminus D_n)+\Vert Du_i\Vert(D_n))
= 0+\liminf_{n\to\infty}\eps/2=\eps/2.
\]
Since $i\ge k$ was arbitrary, we have $D\in \mathcal A_k$, so $\mathcal A_k$ is closed.

Let
\[
\mathcal{Y}:=\{D\subset \Omega:\, \capa_1(D)<\alpha\}.
\]
By \eqref{eq:absolute continuity at limit}, $\mathcal{Y}=\bigcup_{k=1}^{\infty}(\mathcal A_k\cap \mathcal{Y})$.
Since $\mathcal{Y}$ is an open subset of a complete metric space, Baire's category theorem applies.
Thus at least one of the sets $\mathcal A_k$ has nonempty interior in $\mathcal{Y}$. That is, there exists $D\in \mathcal{Y}$ and $\widetilde{\delta}>0$ such that every $H\subset \Omega$ with $\capa_1(H\Delta D)<\widetilde{\delta}$ belongs to $\mathcal A_k$.
Take any $A\subset \Omega$ with $\capa_1(A)<\widetilde{\delta}$. Then
\[
\capa_1((D\cup A)\Delta D)<\widetilde{\delta}
\]
and so
\[
\sup_{i\ge k}\Vert Du_i\Vert(A)\le \sup_{i\ge k}\Vert Du_i\Vert(D\cup A)\le\eps/2< \eps.
\]
By Lemma \ref{lem:absolute cont of variation measure wrt capacity}, we find $\widehat{\delta}>0$ such that if $A\subset\Omega$ with $\capa_1(A)<\widehat{\delta}$, then $\Vert Du_i\Vert(A)<\eps$ for all $i=1,\ldots,k-1$. Finally, we let $\delta:=\min\{\widetilde{\delta},\widehat{\delta}\}$.
\end{proof}

\noindent Address:\\

\noindent Department of Mathematical Sciences\\
4199 French Hall West\\
University of Cincinnati\\
2815 Commons Way\\
Cincinnati, OH 45221-0025\\
P.O. Box 210025\\
E-mail: {\tt panu.lahti@aalto.fi}

\end{document}